\documentclass[leqno,12pt]{amsart}
\usepackage{amsmath,amstext,amssymb,amsopn,amsthm,mathrsfs}

\numberwithin{equation}{section}

\allowdisplaybreaks

\DeclareMathOperator*{\essup}{ess\,sup}

\DeclareMathOperator{\domain}{Dom}
\DeclareMathOperator{\e}{{\mathbb{E}\textrm{xp}}}

\newtheorem{theor}{Theorem}[section]
\newtheorem{propo}[theor]{Proposition}
\newtheorem{lemma}[theor]{Lemma}

\newtheorem*{rem*}{Remark}

\newcommand{\N}{\mathbb{N}^d}
\newcommand{\R}{\mathbb{R}^d}

\begin{document}
\footnotetext{
Research of both authors supported by MNiSW Grant N201 054 32/4285.
}

\title[Riesz transforms for the Dunkl harmonic oscillator]
	{Riesz transforms for the Dunkl harmonic oscillator}


\author[A. Nowak]{Adam Nowak}
\author[K. Stempak]{Krzysztof Stempak}
\address{ Adam Nowak and Krzysztof Stempak,     \newline
      Instytut Matematyki i Informatyki,
      Politechnika Wroc\l{}awska,       \newline
      Wyb{.} Wyspia\'nskiego 27,
      50--370 Wroc\l{}aw, Poland        \vspace{10pt}}
\email{Adam.Nowak@pwr.wroc.pl, Krzysztof.Stempak@pwr.wroc.pl}

\begin{abstract}
We propose an approach to the theory of Riesz transforms in a framework 
emerging from certain reflection symmetries in Euclidean spaces.
Relying on R\"osler's construction of multivariable generalized Hermite polynomials and 
Hermite functions associated with 
a finite reflection group on $\R$, we define and investigate a system of Riesz 
transforms related to the Dunkl
harmonic oscillator. In the case isomorphic with the group $\mathbb{Z}^d_2$ it is proved that 
the Riesz transforms are
Calder\'on-Zygmund operators in the sense of the associated space of homogeneous type,
thus their mapping properties follow from the general theory.
\end{abstract}
\maketitle

\section{Introduction} \label{sec:intro}
In \cite{NS1} the authors developed a unified approach to the theory of Riesz transforms  
in the setting of multi-dimensional orthogonal expansions with a rather general second order
differential operator as the underlying ``Laplacian''. 
It is remarkable that this approach also works within the
framework of differential-difference operators on $\R$ related to finite reflection groups,
the setting intimately connected with the Dunkl theory, which has gained a considerable
interest in various fields of mathematics as well as in theoretical physics during the last years.

Given a finite reflection group $G\subset O(\R)$ and a $G$-invariant nonnegative multiplicity
function $k\colon R\to[0,\infty)$ on a root system $R\subset \R$ associated with the 
reflections of $G$, the Dunkl differential-difference operators $T_j^{k}$, $j=1,\ldots,d$, are
 defined by
$$
T_j^{k}f(x)=\partial_j f(x) +\sum_{\beta\in R_+}k(\beta)\beta_j\frac{f(x)-
	f(\sigma_\beta x)}{\langle\beta,x\rangle}, \quad f\in C^1(\R);
$$
here $\partial_j$ is the $j$th partial derivative, $\langle\cdot,\cdot\rangle$ denotes the Euclidean
inner product in $\R$, $R_+$ is a fixed positive subsystem of $R$, and $\sigma_\beta$ denotes the
reflection in the hyperplane orthogonal to $\beta$. The Dunkl operators $T_j^{k}$, $j=1,\ldots,d$,
form a commuting system (this is an important feature of the system, cf. \cite{Du1}) 
of the first order differential-difference
operators, and reduce to $\partial_j$, $j=1,\ldots,d$, when $k\equiv0$. Moreover,
$T_j^{k}$ are homogeneous of degree $-1$ on $\mathcal{P}$, the space of all polynomials on
$\R$. This means that $T_j^{k}\mathcal{P}_{m}\subset\mathcal{P}_{m-1}$, where
$m\in\mathbb{N}=\{0,1,\ldots\}$ and $\mathcal{P}_{m}$ denotes the subspace of $\mathcal{P}$ 
consisting of polynomials of total degree $m$ (by convention, $\mathcal{P}_{-1}$ consists only
of the null function).

In Dunkl's theory the operator
$$
\Delta_{k}=\sum_{j=1}^d (T_j^{k})^2
$$
plays the role of the Euclidean Laplacian (in fact $\Delta$ comes into play when $k\equiv0$). 
It is homogeneous of degree $-2$ on $\mathcal{P}$ and symmetric in $L^2(\R, w_{k})$, where
$$
w_{k}(x)=\prod_{\beta\in R_+}|\langle\beta,x\rangle|^{2k(\beta)},
$$
if considered initially on $C^\infty_{c,G}(\R)$, 
the space of $C^\infty$ functions on $\R$ with compact support outside the union of all the
hyperplanes $H_\beta$, $\beta\in R_+$; here $H_\beta$ denotes the hyperplane passing through 
the origin and orthogonal to $\beta$. Note that $w_{k}$ is $G$-invariant.

In this article we propose a definition of Riesz transforms associated to the operator 
$$
L_{k}=-\Delta_{k}+\|x\|^2,
$$
which, due to the harmonic confinement $\|x\|^2$,
we call the {\sl Dunkl harmonic oscillator} (it becomes the classic harmonic oscillator
$-\Delta+\|x\|^2$ when $k\equiv0$). In the case of a reflection group isomorphic to the group
$\mathbb{Z}^d_2$ we study $L^p$ mapping properties of the introduced Riesz transforms in detail. 
In the general case it occurs that $L_{k}$ (or rather its self-adjoint extension
$\mathcal{L}_{k}$) has a discrete spectrum and the corresponding eigenfunctions are the 
generalized Hermite functions defined and investigated by R\"osler \cite{R2}. 
Then, under some mild assumptions,
the formal definition $R^{k}_{j} =  \delta_j ({\mathcal{L}_{k}})^{-1 \slash 2}$, rewritten in
terms of the related expansions, delivers an $L^2$-bounded operator.

The results of the present paper naturally extend those established in \cite{ST1} by J.L. Torrea 
and the second author. On the other hand, the results are closely related to the authors articles
\cite{NS2}, \cite{NS3}, where Riesz transforms for multi-dimensional Laguerre function 
expansions were defined 
and thoroughly studied. For basic facts concerning Dunkl's theory we refer the reader to the
excellent survey article by R\"osler \cite{R2}.
There, one can also find a discussion (see \cite[Section 3]{R2}) and extensive 
references concerning
applications of Dunkl's theory in mathematical physics.

The organization of the paper is the following. 
In Section \ref{sec:prel} we define Riesz transforms in the context of the harmonic
oscillator based on a general Dunkl operator.
Section \ref{sec:Z2d} introduces the particular Dunkl setting related to the
group $\mathbb{Z}_2^d$. Here, apart from gathering basic facts, we establish a new result
(Theorem \ref{tmax}) concerning the heat semigroup maximal operator.
In Section \ref{sec:main} the $\mathbb{Z}_2^d$ Riesz-Dunkl transforms and the relevant
kernels are defined, and the main results of the paper are stated (Theorem \ref{thm:main}).
As a typical application of $L^p$-boundedness of Riesz transforms, \emph{a priori}
$L^p$-bounds in the $\mathbb{Z}_2^d$ context are then derived (Proposition \ref{apriori}).
Finally, Section \ref{sec:kernel} is devoted to the proofs of all necessary 
kernel estimates related to the $\mathbb{Z}_2^d$ setting. 

Throughout the paper we use a fairly standard notation. 
Given a multi-index $n\in\N$, we write $|n|=n_1+\ldots+n_d$, $n!=n_1!\cdot\ldots\cdot n_d!$
and, for $x,y\in\R$, $x^n=x_1^{n_1}\cdot\ldots\cdot x_d^{n_d}$, $xy = (x_1 y_1,\ldots,x_d y_d)$; 
$\|x\|$ denotes the Euclidean norm of $x\in\R$, and $e_j$ the $j$th coordinate vector in $\R$.
Given $x\in \R$ and $r>0$, $B(x,r)$ is the Euclidean ball in $\R$ centered at $x$ and of radius $r$.
For a nonnegative weight function $w$ on $\R$, by $L^p(\R,w)$, $1\le p<\infty$, we denote the 
usual Lebesgue spaces related to the measure $dw(x)=w(x)dx$ 
(in the sequel we will often abuse slightly the notation and use the same symbol $w$ to 
denote the measure induced by a density $w$). 
Writing $X\lesssim Y$ indicates that $X\leq CY$ with a positive constant $C$
independent of significant quantities. We shall write $X \simeq Y$ when $X \lesssim Y$
and $Y \lesssim X.$

\section{The general setting} \label{sec:prel}

Similarly to numerous frameworks discussed in the literature (see, for instance, the setting 
of deformed Fock spaces discussed by Lust-Piquard \cite{LP}) it is reasonable to define, 
at least formally, the Riesz transform $\mathcal{R}^k=(R_1^{k},\ldots,R_d^{k})$ associated 
with $L_{k}=-\Delta_{k}+\|x\|^2$ as
\begin{equation}   \label{a1}
  R^{k}_{j} =  \delta_j ({\mathcal{L}_{k}})^{-1 \slash 2},
\end{equation}
where $\mathcal{L}_{k}$ is a suitable self-adjoint extension in $L^2(\R,w_{k})$ of $L_{k}$, and
$\delta_j$'s are appropriately defined first order differential-difference operators. 
In the present setting we 
define the $j$th partial derivative $\delta_j$ related to ${L}_{k}$ by
$$
\delta_j = T_j^{k} + x_j.
$$
Taking into account the fact that $w_{k}$ is $G$-invariant, a short calculation shows
that the (formal) adjoint of $\delta_j$ in $L^2(\R,w_{k})$ is
$$
\delta^{*}_j = -T_j^{k}+ x_j.
$$
To be precise, this means that 
\begin{equation} \label{a1sym}
\langle\delta_j f,g\rangle_{L^2(\R, w_{k})}=\langle f,\delta_j^*g\rangle_{L^2(\R, w_{k})}, 
\qquad f,g\in C^\infty_{c,G}(\R).
\end{equation}
One of the facts which motivate the definition \eqref{a1} is that, as a direct computation shows,
\begin{equation} \label{Ldecomp}
L_{k} = \frac12\sum^{d}_{j=1}\Big( \delta^{*}_j \delta_j+ \delta_j \delta_j^{*}\Big).
\end{equation}

In the setting of general Dunkl's theory R\"osler \cite{R1} constructed systems of 
naturally associated multivariable generalized Hermite polynomials and Hermite functions. 
The construction starts from an arbitrarily chosen orthonormal basis $\{\varphi_n\}$ of 
$\mathcal{P}$ equipped with the generalized Fischer inner product
$$
[p,q]_k=\big(p(T^k)q\big)(0), \qquad p,q\in\mathcal{P},
$$
such that $\varphi_n\in \mathcal{P}_{|n|}$ and the coefficients of $\varphi_n$'s are real 
(more precisely,
we should consider $\{\varphi_n\}$ to be a basis of the completion of $\mathcal{P}$). 
Above, $p(T^k)$ is understood as the differential-difference operator derived from 
$p(x)$ by replacing $x_j$ by $T_j^k$.
It is clear that $\mathcal{P}_{m_1} \perp_k \mathcal{P}_{m_2}$ for $m_1\neq
m_2$. For polynomials $p\in \mathcal{P}_{m_1},q\in \mathcal{P}_{m_2}$, $m_1,m_2\in\mathbb{N}$, 
the following identity holds (see \cite[(3.1)]{R1}, actually we use a slight modification of it)
\begin{equation}\label{fund}
[p,q]_k=c^{-1}_{k}2^{(m_1+m_2)/2}\int_{\R}\exp\Big(-\frac{\Delta_{k}}4\Big)p(x)
\exp\Big(-\frac{\Delta_{k}}4\Big)q(x)e^{-\|x\|^2}w_{k}(x)\,dx, 
\end{equation} 
where the constant $c_k$ is the so-called Macdonald-Mehta-Selberg integral,
$$
c_{k}=\int_{\R}\exp(-\|x\|^2)w_{k}(x)\,dx
$$
(note that for any $p\in\mathcal{P}$ the series defining $\exp(-\Delta_{k}/4)p$ terminates).

The generalized Hermite polynomials $H_n^{k}$, $n\in\N$, are then defined by
$$
H_n^{k}(x)=2^{|n|}\sqrt{n!}\exp(-\Delta_{k}/4) \varphi_n(x).
$$
Note that the definition depends essentially on the choice of the system $\{\varphi_n\}$ 
which is not specified above.
Orthogonality of $\varphi_n$'s and \eqref{fund} then show that $\{H_n^{k} : n\in\N\}$, 
is an orthogonal system in $L^2(\R, e^{-\|\cdot\|^2}w_{k})$, while the system of 
generalized Hermite functions
$$
h_n^{k}(x)=\Big(\frac{1}{2^{|n|}n!c_{k}}\Big)^{1/2}\exp(-\|x\|^2/2)H_n^{k}(x),
\qquad x\in \R,\quad n\in \N,
$$
is an orthonormal basis in $L^2(\R, w_{k})$, cf. \cite[Corollary 3.5 (ii)]{R1} 
(note a slightly different normalization: adjusting the factor $\sqrt{n!}$ in the definition
of $H_n^k$ and the canonical choice of the basis $\varphi_n(x)=(n!)^{-1/2}x^n$
results in obtaining the classical Hermite polynomials  
for $k\equiv0$,
whereas the coefficient in the definition of $h_n^k$ normalizes the system in $L^2(\R,w_k)$).
Moreover,  
$h_n^k$ are eigenfunctions of  $L_{k}$, 
$$
L_{k}h_n^{k}=(2|n|+2\gamma+d)h_n^{k},
$$
where $\gamma=\sum_{\beta\in R_+}k(\beta)$. For $k\equiv0$ and $\varphi_n(x)=(n!)^{-1/2}x^n$, $H_n^{0}(x)=\prod_{j=1}^dH_{n_j}(x_j)$,
where $H_{n_j}$, $n_j\in\mathbb{N}$, denote the classical Hermite polynomials of degree $n_j$,
cf. \cite[Example 3.3 (1)]{R1}; similarly, $h_n^{0}$ are the usual multi-dimensional 
Hermite functions,
cf., for instance, \cite{ST1}.

Let $\langle\cdot,\cdot\rangle_{k}$ be the canonical inner product in $L^2(\mathbb{R}^d, w_{k})$,
$$
\langle f, g\rangle_{k}=\int_{\mathbb{R}^d}f(x)\overline{g(x)}\,w_{k}(x)\,dx,
$$
and $\|\cdot\|_{k}$ be the norm induced by $\langle\cdot,\cdot\rangle_{k}$.
The operator
$$
  \mathcal{L}_{k} f=\sum_{n\in \N} (2|n|+2\gamma+d)\langle f,h_n^{k}
  \rangle_{{k}} \, h_n^{k}, 
$$
defined on the domain
$$
  \domain(\mathcal{L}_{k})=\Big\{f\in L^2(\R, w_{{k}}):
  \sum_{n\in \N}|(2|n|+2\gamma+d)\langle f,h_n^k
  \rangle_{k}|^2<\infty\Big\},
$$
is a self-adjoint extension of $L_{k}$ considered on $C^{\infty}_{c,G}(\R)$
as the natural domain (the inclusion $C^\infty_{c,G}(\R)\subset
\domain(\mathcal{L}_{k})$ may be easily verified). 
The spectrum of $\mathcal{L}_k$ is the discrete set
$\{2m+2\gamma+d:m\in \mathbb{N}\}$, and
the spectral decomposition of $\mathcal{L}_k$ is 
$$
\mathcal{L}_{k} f=\sum_{m=0}^\infty(2m+2\gamma+d)\mathcal{P}_m^k f, \qquad
f\in\domain(\mathcal{L}_{k}),
$$
where the spectral projections are
\begin{equation*} 
\mathcal{P}_m^{k} f=\sum_{|n|=m}\langle f,h_n^{k}
  \rangle_{{k}} \,h_n^{k}. 
\end{equation*}

Since the spectrum of $\mathcal{L}_{k}$ is separated from zero, 
$$
\mathcal{L}_{k}^{-1/2}f=\sum_{m=0}^\infty (2m+2\gamma+d)^{-1/2}\mathcal{P}_m^{k} f
$$
is a bounded operator on $L^2(\R, w_{k})$. 
It would be tempting to furnish the rigorous definition of $R^{k}_j$ 
as an $L^2$-bounded operator by writing
\begin{equation} \label{RR}
R^{k}_jf=\sum_{n\in\N} (2|n|+2\gamma+d)^{-1/2}\langle f,h_n^{k} \rangle_{k}\; \delta_jh_n^{k},
\end{equation}
for $f$ with the expansion $f=\sum_{n\in\N}\langle f,h_n^{k}\rangle_{k}h_n^{k}$.
This is indeed possible 
when two additional hypotheses are satisfied: 
\begin{itemize}
\item[(1)] the system $\{\delta_jh_n^{k} : n \in \N\}$, 
after excluding null functions,
is an orthogonal system in $L^2(\R, w_{k})$; 
\item[(2)] $\|\delta_jh_n^{k}\|_{k}\lesssim (|n|+1)^{1/2}$, $n\in\N$.
\end{itemize}
It occurs that the required properties are consequences of hypotheses that can be imposed on
$\{\varphi_n : n\in \N\}$.
For the purpose of the proposition that follows, for $j=1,\ldots,d$, we denote by
$\widehat{\mathbb{N}}^d_j$ the set of those $n\in \N$ for which $T_j^k\varphi_n$ does not 
vanish identically; also, by $\|\cdot\|_{[\,,\,]_k}$ we denote the norm in $\mathcal{P}$ induced 
by $[\cdot,\cdot]_k$.
\begin{propo} \label{eldwa}
Let $\{\varphi_n : n\in \N\}$ be an orthonormal basis in $\mathcal{P}$ such that $\varphi_n\in
\mathcal{P}_{|n|}$ and the coefficients of $\varphi_n$'s are real. 
Assume, in addition, that for every $j=1,\ldots,d$,
\begin{itemize}
\item[(i)] the system $\{T_j^k\varphi_n : n\in\widehat{\mathbb{N}}^d_j\}$ is orthogonal in
$(\mathcal{P},[\cdot,\cdot]_k)$;
\item[(ii)] $\|T_j^k\varphi_n \|_{[\,,\,]_k}\lesssim (|n|+1)^{1/2}$, $n\in\widehat{\mathbb{N}}^d_j$.
\end{itemize}
Then, for every  $j=1,\ldots,d$, $R_j^k$ defined by \eqref{RR} is bounded on $L^2(\R,w_k)$.
\end{propo}
\begin{proof} 
Assumptions which we impose on $\{T_j^k\varphi_n\}$ imply the corresponding properties of
$\{\delta_jh_n^{k}\}$ listed above. Indeed, since $T_j^k$ commutes with $\Delta_k$, thus also with
$\exp(-\Delta_k/4)$, and $T_j^k(fg)=(T_j^kf)g+(T_j^kg)f$ provided $f$ is $G$-invariant, we have
\begin{equation*}
\delta_jh_n^{k}=\big({c_{k} 2^{|n|}n!}\big)^{-1/2} e^{-\|x\|^2\slash 2}\, T_j^kH_n^k
=\big(2^{|n|}\slash {c_{k}}\big)^{1/2}e^{-\|x\|^2\slash 2} \exp(-\Delta_k/4)(T_j^k\varphi_n).
\end{equation*}
Therefore (1) and (2) follow by applying \eqref{fund} to the polynomials $T_j^k\varphi_n$. 
In consequence
the conclusion follows by Bessel's inequality applied to the orthogonal system 
$\{T_j^k \varphi_n : n\in\widehat{\mathbb{N}}^d_j\}$.
\end{proof}
A comment is probably in order, concerning a possibility of constructing an orthonormal basis 
with properties as in the first sentence of Proposition \ref{eldwa} and satisfying, in addition,  
at least Condition (i) of that proposition. As indicated in \cite[Section 3]{R1}, since
 $\mathcal{P}_{m_1} \perp_k \mathcal{P}_{m_2}$ for $m_1\neq
m_2$, $m_1, m_2\in \mathbb{N}$, the desired basis $\{\varphi_n\}$ can be constructed by 
Gram-Schmidt orthogonalization within every $\mathcal{P}_{m}$, $m\in \mathbb{N}$, from an 
arbitrarily ordered real-coefficient basis of  $\mathcal{P}_{m}$. To guarantee fulfillment 
of Condition (i) in Proposition \ref{eldwa}, again it is sufficient to act 
with every $m\ge1$ separately; this is because $T_j^{k}\mathcal{P}_{m}\subset\mathcal{P}_{m-1}$.
Summing up, 
given $m\ge1$, an orthonormal real-coefficient basis $\{\varphi_n: |n|=m\}$ of $\mathcal{P}_{m}$ 
should be chosen in such a way that for every  $j=1,\ldots,d$, the system 
$\{T_j^k\varphi_n: |n|=m\}$, after excluding null functions, is orthogonal in $\mathcal{P}_{m-1}$.
Note that this means that for every   $j=1,\ldots,d$, a number of functions 
$T_j^k\varphi_n$, $|n|=m$, must vanish.

We mention at this point that in the case of the group $\mathbb{Z}^d_2$ discussed in detail 
in the subsequent sections, 
the choice of appropriately normalized monomials $x^n$, $n\in\N$, leads to an orthonormal
basis in $\mathcal{P}$ satisfying the assumptions of Proposition \ref{eldwa}. This may be verified
directly (see a comment in Section \ref{sec:Z2d}), 
but taking into account the explicit form of the generalized Hermite functions 
that appear, it may be checked as well that the properties (1) and (2) preceding
Proposition \ref{eldwa} are also satisfied (see Section \ref{sec:main} below).
Whether or not such a choice of $\varphi_n$'s may be done for an arbitrary reflection group
remains an open question.

Another remark that seems to be worth of mentioning is that it would be tempting to
justify the boundedness on $L^2(\mathbb{R}^d,w_k)$ of $R^k_j$ given by \eqref{a1}, for an arbitrary
reflection group $G$, without requirements (1) and (2), but instead with using an argument based merely
on \eqref{a1}, \eqref{a1sym} and the self-adjointness of $\mathcal{L}_k$.
With the notation $\widetilde{R}_j^k = \delta_j^{*} \mathcal{L}_k^{-1\slash 2}$
the argument would proceed as follows:
\begin{align*}
\|R_j^k\|^2_{L^2} & \le \|R_j^k f\|^2_{L^2} + \|\widetilde{R}_j^k f\|^2_{L^2} \\
	& = \big\langle \delta_j^{*} \delta_j \mathcal{L}_k^{-1\slash 2} f, \mathcal{L}_k^{-1\slash 2}f 
		\big\rangle_{k} + \big\langle \delta_j \delta_j^{*} 
		\mathcal{L}_k^{-1\slash 2} f, \mathcal{L}_k^{-1\slash 2}f \big\rangle_{k} \\
	& \le \Big\langle \Big(\sum_{i=1}^d \big( \delta_i^{*} \delta_i + \delta_i \delta_i^{*} \big)
		\mathcal{L}_k^{-1\slash 2} f \Big), \mathcal{L}_k^{-1\slash 2} f \Big\rangle_k \\
	& = 2 \big\langle L_k \mathcal{L}_k^{-1\slash 2} f, \mathcal{L}_k^{-1\slash 2} f \big\rangle_k
	= 2 \|f\|_{L^2}^2.
\end{align*}
Note, however, that the use of \eqref{a1sym} is limited to a special class of functions
that would require $f$ from a dense subspace of $L^2(\mathbb{R}^d,w_k)$ 
(actually from a subspace of $\domain(\mathcal{L}_k^{-1\slash 2})$) to satisfy:
$\mathcal{L}_k^{-1\slash 2}f, \delta_j\mathcal{L}_k^{-1\slash 2} f \in C^{\infty}_{c,G}(\mathbb{R}^d)$.
It seems to be far from obvious to determine whether such a dense subspace always exists.

\section{Preliminaries to the $\mathbb{Z}_2^d$ group case} \label{sec:Z2d}

Consider the finite reflection group generated by $\sigma_j$, $j=1,\ldots,d$,
$$
\sigma_{j}(x_1,\ldots,x_j,\ldots,x_d)=(x_1,\ldots,-x_j,\ldots,x_d),
$$
and isomorphic to $\mathbb{Z}_2^d = \{0,1\}^d$.
The reflection $\sigma_j$ is in the hyperplane orthogonal to $e_j$, the $j$th coordinate 
vector in $\R$. Thus $R=\{\pm \sqrt{2}e_j: j=1,\ldots,d\}$, $R_+=\{\sqrt{2}e_j: j=1,\ldots,d\}$, 
and for a nonnegative multiplicity function $k\colon R\to [0,\infty)$ which is
$\mathbb{Z}^d_2$-invariant only values of $k$ on $R_+$ are essential. Hence we may think 
$k=(\alpha_1+1\slash 2,\ldots,\alpha_d+1\slash 2)$, 
$\alpha_j\ge-1/2$. We write $\alpha_j+1\slash 2$ in place of seemingly more appropriate $\alpha_j$
since, for the sake of clarity, it is convenient for us to stick to the notation used 
in \cite{NS3}.

In what follows the symbols $T_j^\alpha$, $\delta_j$, $\Delta_\alpha$, $w_\alpha$, $L_\alpha$,
$h_n^\alpha$, $[\cdot,\cdot]_{\alpha}$ and so on, denote the objects introduced in Section \ref{sec:prel} and related to 
the present setting. Thus
the Dunkl differential-difference operators $T_j^\alpha$, $j=1,\ldots,d$, are now given by 
$$
T_j^{\alpha}f(x)=\partial_j f(x) +(\alpha_j+1\slash 2)\frac{f(x)-f(\sigma_j x)}{x_j}, 
\qquad f\in C^1(\R),
$$
and the explicit form of the Dunkl Laplacian is
$$
\Delta_\alpha f(x)=
\sum_{j=1}^d\bigg(\frac{\partial^2f}{\partial x_j}(x)+\frac{2\alpha_j+1}{x_j}\frac{\partial
f}{\partial x_j}(x)-(\alpha_j+1\slash 2)\frac{f(x)-f(\sigma_jx)}{x_j^2}\bigg).
$$
Note that $\Delta_\alpha$, when restricted to the subspace
$$
C^1(\mathbb{R}^d)_R=\{f\in C^1(\mathbb{R}^d): \forall j=1,\ldots,d,\,\,\,f(x)=f(\sigma_j x)\},
$$
coincides with the multi-dimensional Bessel differential operator
$\sum_{j=1}^d(\partial^2_j+\frac{2\alpha_j+1}{x_j}\partial_j)$,
and 
$L_\alpha=-\Delta_\alpha+\|x\|^2$ reduces to
$$
-\Delta + \|x\|^2 - \sum_{j=1}^d \frac{2\alpha_j+1}{x_j} \frac{\partial}{\partial x_j},
$$
the operator investigated in \cite{NS3} 
(to be precise, both operators are assumed to act on functions defined on 
$\mathbb{R}^d_+ = (0,\infty)^d$). 

The corresponding weight $w_\alpha$ has the form
$$
w_\alpha(x)= \prod_{j=1}^d|x_j|^{2\alpha_j+1} \simeq
\prod_{\beta \in R_+} |\langle \beta, x \rangle_{\alpha}|^{2k(\beta)}, \qquad x\in \R;
$$
here $\langle\cdot,\cdot\rangle_\alpha$ denotes the inner product in $L^2(\mathbb{R}^d, w_\alpha)$.
In dimension one (see \cite[Example 3.3 (2)]{R1}) for the reflection group $\mathbb{Z}_2$ and 
the multiplicity parameter $\alpha+1\slash 2$, $\alpha\ge-1/2$, the polynomial basis 
$\{\varphi_n\}$ (orthonormal with respect to $[\cdot,\cdot]_{\alpha}$) is determined 
uniquely (up to sign changes) by suitable normalization of the monomials $x^n$, 
$n\in \mathbb{N}$. 
Let $a_{n,\alpha}$ be the sequence determined by the following recurrence relation: for $n\ge1$,
$a_{n,\alpha}=na_{n-1,\alpha}$ if $n$ is even and $a_{n,\alpha}=(n+2\alpha+1)a_{n-1,\alpha}$ 
if $n$ is odd, $a_{0,\alpha}=1$. It is easily checked that 
$\varphi_n^\alpha(x)=a_{n,\alpha}^{-1/2}x^n$ is an orthonormal basis. Moreover, one obtains as 
the corresponding generalized Hermite polynomials 
the {\sl genuine} generalized Hermite polynomials $H_n^{\alpha+1/2}$ on $\mathbb{R}$,  
as defined and  studied in \cite{Ch}, see also \cite{R} 
(this is the only inconsistency in our notation: writing $H_n^\alpha$ would 
collide with the already accepted notation). 
The system is orthogonal in $L^2(\mathbb{R},|x|^{2\alpha+1}e^{-x^2}\,dx)$, and
\begin{align}
\label{HeLa} H_{2n}^{\alpha+1/2}(x)&=(-1)^n2^{2n}n!L_n^{\alpha}(x^2),\\
H_{2n+1}^{\alpha+1/2}(x)&=(-1)^n2^{2n+1}n!xL_n^{\alpha+1}(x^2)\nonumber;
\end{align}
here for $\alpha>-1$ and $n\in\mathbb{N}$, 
$L^\alpha_n$ denotes the Laguerre polynomial of degree $n$ and order
$\alpha,$ see \cite[p.\,76]{Leb}. Note that when $\alpha=-1/2$, 
$H_n^0$ is indeed the classical Hermite polynomial, see \cite[p.\,81]{Leb}.
The $d$-dimensional setting emerges through a tensorization one-dimensional settings. 
Thus, given $\alpha = (\alpha _1, \ldots , \alpha _d) \in [-1/2,\infty)^{d}$, 
and taking $\varphi_n^\alpha(x)=a_{n,\alpha}^{-1/2}x^n$, $n\in\N$, $a_{n,\alpha}=\prod_{i=1}^d a_{n_i,\alpha_i}$,
the corresponding generalized Hermite functions are given by
$$
h_{n}^{\alpha}(x) = h _{n_1}^{\alpha _1}(x_1) \cdot \ldots \cdot
h_{n_d}^{\alpha _d}(x_d), \qquad x = (x_1, \ldots ,x_d)\in \R,\quad n=(n_1,\ldots,n_d) \in \N;
$$
here $h_{n_i}^{\alpha _i}$ are the one-dimensional functions
\begin{equation*}
h_{n_i}^{\alpha_i}(x_i) =c_{n_i,\alpha_i}e^{-x_i^2/2}H_{n_i}^{\alpha_i+1/2}(x_i),
\end{equation*}
with the normalization constants $c_{n_i,\alpha_i}$, where
$c_{2n_i,\alpha_i}=(2^{4n_i}\Gamma(n_i+1)\Gamma(n_i+\alpha_i+1))^{-1/2}$ and
$c_{2n_i+1,\alpha_i}=(2^{4n_i+2}\Gamma(n_i+1)\Gamma(n_i+\alpha_i+2))^{-1/2}$.
In view of \eqref{HeLa}, one has 
\begin{align*}
h_{2n_i}^{\alpha_i}(x_i)&=d_{2n_i,\alpha_i}e^{-x_i^2/2}L_{n_i}^{\alpha_i}(x_i^2),\\
h_{2n_i+1}^{\alpha _i}(x_i)&=d_{2n_i+1,\alpha_i}e^{-x_i^2/2}x_iL_{n_i}^{\alpha_i+1}(x_i^2),
\end{align*}
where
$$ 
d_{2n_i,\alpha_i}=(-1)^{n_i}\bigg(\frac{\Gamma(n_i+1)}{\Gamma(n_i+\alpha_i +1)}\bigg)^{1/2}, \qquad
d_{2n_i+1,\alpha_i}=(-1)^{n_i}\bigg(\frac{\Gamma(n_i+1)}{\Gamma(n_i+\alpha_i +2)}\bigg)^{1/2}.
$$ 

Let us comment at this point that $\{\varphi_n^{\alpha} : n \in \N\}$ satisfies assumptions
(i), (ii) of Proposition \ref{eldwa}. Indeed, 
$T_j^{\alpha} \varphi_n^{\alpha} = b_{n,\alpha} x^{n-e_j}$, where $b_{n,\alpha}$ equals
$a_{n,\alpha}^{-1\slash 2}$ multiplied either by $n_j$ or $n_j+2\alpha_j+1$, depending on
whether $n_j$ is even or odd, respectively. Hence (i) holds.
On the other hand, $\|T_j^{\alpha} \varphi_n^{\alpha}\|_{[\cdot,\cdot]_{\alpha}}
= b_{n,\alpha} a_{n-e_j,\alpha}^{1\slash 2} = \mathcal{O}(|n|^{1\slash 2})$, which implies (ii).

The system $\{h_{n}^{\alpha} : n \in \N \}$ is an orthonormal basis in $L^2(\R,w_{\alpha})$
consisting of eigenfunctions of  $L_\alpha$,
$$
L_\alpha h_n^{\alpha}=(2|n|+2|\alpha |+2d)h_n^{\alpha},
$$
where by $|\alpha|$ we denote $|\alpha|=\alpha_1+\ldots+\alpha_d$ (thus $|\alpha|$ may be negative).
For $\alpha=(-1/2,\ldots,-1/2)$ we obtain the usual Hermite functions. 
Recall that, considered on appropriate domain, the operator $L_\alpha$ is positive and symmetric 
in $L^2(\R, w_{\alpha})$.
The spectral decomposition of $\mathcal{L}_{\alpha}$ is
$$
\mathcal{L}_\alpha f=\sum_{m=0}^\infty(2m+2|\alpha|+2d)\mathcal{P}_m^\alpha f, \qquad
f\in\domain(\mathcal{L}_\alpha),
$$
with the spectral projections
\begin{equation*} 
\mathcal{P}_m^\alpha f=\sum_{|n|=m}\langle f,h_n^{\alpha}
  \rangle_{\alpha}\, h_n^{\alpha}, \qquad m \in \mathbb{N}. 
\end{equation*}

The semigroup $T_t^\alpha = \exp(-t\mathcal{L}_{\alpha})$, $t \ge 0$, generated by
$\mathcal{L}_{\alpha}$ is a strongly continuous semigroup of contractions on
$L^2(\R, w_{\alpha})$. By the spectral theorem,
\begin{equation} \label{sr_t}
T_t^\alpha f=\sum_{m=0}^\infty e^{-t(2m+2|\alpha|+2d)}\mathcal{P}^\alpha_mf, 
\qquad f\in L^2(\R, w_{\alpha}).
\end{equation}
The integral representation of $T_t^\alpha$ on $L^2(\R,w_{\alpha})$ is
\begin{equation} \label{ir_t}
  T_t^\alpha f(x)=\int_{\R} G_t^\alpha(x,y)f(y)\,dw_{\alpha}(y),
  \qquad x\in \R,
\end{equation}
where the heat kernel is given by
\begin{equation} \label{jc}
G^\alpha_t(x,y)=\sum_{m=0}^\infty e^{-t(2m+2|\alpha|+2d)} \sum_{|n|=m}
h_n^\alpha(x)h_n^\alpha(y).
\end{equation}

In dimension one, for $\alpha\ge-1/2$ it is known (see, for instance, 
\cite[Theorem 3.12]{R1} and \cite[p.\,523]{R1})
that
$$
G^\alpha_t(x,y)=\frac{1}{(2\sinh 2t)^d}\exp\Big({-\frac{1}{2} \coth(2t)\big(x^{2}+y^{2}\big)}\Big)
\Bigg[\frac{I_{\alpha}\left(\frac{x y}{\sinh 2t}\right)}
{(x y)^{\alpha}}+xy \frac{I_{\alpha+1}\left(\frac{x y}{\sinh 2t}\right)}{(x y)^{\alpha+1}}\Bigg],
$$
with $I_\nu$ being the modified Bessel function of the first kind and order $\nu$,
$$
I_{\nu}(z) = \sum_{k=0}^{\infty} \frac{(z\slash 2)^{\nu+2k}}{\Gamma(k+1)\Gamma(k+\nu+1)}, 
	\qquad |\arg z| < \pi;
$$
$I_{\nu}$, as a function on $\mathbb{R}_{+}$,
is real, positive and smooth for any $\nu > -1$, see \cite[Chapter 5]{Leb}.
Notice that the ratio $I_{\nu}(z)\slash z^{\nu}$ occurring in the formula for $G_t^{\alpha}(x,y)$
represents the entire function of $z$.

Therefore, in $d$ dimensions,
$$
G^\alpha_t(x,y)=\sum_{\varepsilon\in\{0,1\}^d}G^{\alpha,\varepsilon}_t(x,y),
$$
where
$$
G^{\alpha,\varepsilon}_t(x,y)=\frac{1}{(2\sinh 2t)^{d}}\exp\Big({-\frac{1}{2}
 \coth(2t)\big(\|x\|^{2}+\|y\|^{2}\big)}\Big) \prod^{d}_{i=1} (x_i y_i)^{\varepsilon_i}
  \frac{I_{\alpha_i+\varepsilon_i}\left(\frac{x_i y_i}{\sinh 2t}\right)}{(x_i
   y_i)^{\alpha_i+\varepsilon_i}}.
$$
By means of the standard asymptotics for $I_{\nu}$ (see \cite[(5.16.4), (5.16.5)]{Leb})
it follows that the integral in \eqref{ir_t} converges for any $f \in L^p(\R,w_{\alpha})$,
$1\le p < \infty$, and therefore it may serve as a definition of $T_t^{\alpha}f$ on these spaces.

Note that for $\varepsilon_{o} = (0,\ldots,0)$ the component kernel $G_t^{\alpha,\varepsilon_o}(x,y)$
coincides, up to the factor $2^{-d}$, with the heat kernel associated with Laguerre function
expansions studied in \cite{NS3}.
Moreover, by Soni's inequality \cite{Soni},
$$
I_{\nu+1}(z) < I_{\nu}(z), \qquad z>0, \quad \nu \ge -\frac{1}{2},
$$
it follows immediately that for $\alpha \in [-1\slash 2,\infty)^d$
$$
0 < G_t^{\alpha}(x,y) \lesssim G_t^{\alpha,\varepsilon_{o}}(x,y), \qquad t>0, \quad x,y \in \R.
$$
Then, as a consequence of \cite[Theorem 2.1]{NS3}, 
we get an interesting result on the maximal
operator $T_*^{\alpha}f = \sup_{t>0}|T_t^{\alpha}f|$, which may be of independent interest
(for the definition of the weight classes $A_p^{\alpha}$, see Section \ref{sec:main} below).

\begin{theor} \label{tmax}
Let $\alpha \in [-1\slash 2, \infty)^d$. Then $T^{\alpha}_*$ is bounded
on $L^p(\R,W dw_{\alpha})$, $W \in A^{\alpha}_p$, $1<p<\infty$, and from
$L^1(\R, W dw_{\alpha})$ to $L^{1,\infty}(\R,W dw_{\alpha})$, $W \in A^{\alpha}_1$.
\end{theor}

The proof of Theorem \ref{tmax} reduces to an application of the abovementioned result 
from \cite{NS3}. To see that this is possible, one has to take into account that
$G_t^{\alpha,\varepsilon_{o}}(x,y)$ is even with respect to each coordinate of $(x,y)$ 
and that $w_{\alpha}(x)$ has the analogous property.
Also, it is necessary to know that, given $\alpha \in [-1\slash 2,\infty)^d$
and $1\le p < \infty$, the condition $W \in A^{\alpha}_p$ is equivalent to the condition
$W_\xi \in A_p^{\alpha,+}$ for all $\xi \in \{-1,1\}^d$, where
$W_\xi$ denotes the weight on $\R_{+}$ given by $W_\xi (x) = W(\xi x)$,
and $A_p^{\alpha,+}$ stands for the Muckenhoupt class of weights associated with the 
homogeneous space $(\R_{+},w^+_{\alpha},\|\cdot\|)$, $w^+_{\alpha}$ being the restriction
of $w_{\alpha}$ to $\R_{+}$; see (1.3), (1.4) and related comments in \cite{NS3}.
We leave further details to an interested reader.

Although in the present section we focused on the $\mathbb{Z}_2^d$ case, it should be pointed
out that there is a general background for the results considered here for an arbitrary reflection
group, see \cite{R2} for a comprehensive account.
In particular, the heat (or Mehler) kernel \eqref{jc} has always a closed form involving the
so-called Dunkl kernel, and is always strictly positive.
This implies that the corresponding semigroup is contractive on $L^{\infty}(\mathbb{R}^d,w_k)$,
and as its generator is self-adjoint and positive in $L^2(\mathbb{R}^d,w_k)$, the semigroup
is also contractive on the latter space. Hence, by duality and interpolation, it is in fact
contractive on all $L^p(\mathbb{R}^d,w_k)$, $1\le p \le \infty$.

\section{$\mathbb{Z}_2^d$ Riesz transforms} \label{sec:main}

First of all, we shall see how $\delta_j$'s act on $h^\alpha_n$. It is sufficient to consider 
the one-dimensional situation and then distinguish between the even and odd cases. Recall that
$\delta_j=T_j^\alpha+x_j$; in the one-dimensional case we simply write $\delta$ in place of
$\delta_1$. For $n\in\mathbb{N}$ and $\alpha\ge-1/2$, combining the fact that $h^\alpha_{2n}$ 
is an even function with the identity 
\begin{equation}\label{fm}
\frac d{dx}L^\alpha_n=-L^{\alpha+1}_{n-1}, \qquad n \ge 1,
\end{equation}
see \cite[(4.18.6)]{Leb}, one easily obtains
$$
\delta h^\alpha_{2n}=\sqrt{4n} h^\alpha_{2n-1}.
$$
Similarly, using the fact that $h^\alpha_{2n+1}$ is an odd function, \eqref{fm} and the identities
\begin{equation} \label{nn2}
-yL^{\alpha+2}_{n-1}(y)+(\alpha+1)L^{\alpha+1}_{n-1}(y)=nL^\alpha_n(y), \qquad n \ge 1,
\end{equation}
which can be deduced from \eqref{fm} and \cite[(4.18.2), (4.18.7)]{Leb}, and
\begin{equation} \label{nn3}
L^{\alpha+1}_{n-1}-L^{\alpha+1}_n=-L^\alpha_n, \qquad n \ge 1
\end{equation}
(this is \cite[(4.18.5)]{Leb} after correcting a misprint in one of 
the superscripts), one gets
$$
\delta h^\alpha_{2n+1}=\sqrt{4n+4\alpha+4}h^\alpha_{2n}.
$$

Summarizing, in $d$ dimensions, for $n\in\N$ and $\alpha\in[-1/2,\infty)^d$ we have
\begin{equation} \label{nn1}
\delta_jh^\alpha_{n}=m(n_j,\alpha_j)h^\alpha_{n-e_j},
\end{equation}
where
$$
m(n_j,\alpha_j)= \left\{ \begin{array}{ll}
\sqrt{2n_j}, & \textrm{if $n_j$ is even,} \\
\sqrt{2n_j+4\alpha_j+2}, & \textrm{if $n_j$ is odd};
\end{array} \right.
$$
here, and also later on, we use the convention that $h_{n-e_j} \equiv 0$ if $n_j=0$.
Note that for $\alpha = (-1\slash 2,\ldots, -1\slash 2)$ 
this is consistent with \cite[(3.2)]{ST1}. 
Notice also that $\{\delta_j h_n^{\alpha} : n_j \ge 1\}$ is an orthogonal system in
$L^2(\R,w_{\alpha})$ and $\|\delta_j h_n^{\alpha}\| = \mathcal{O}(|n|^{1\slash 2})$,
hence the hypotheses (1) and (2) in Section \ref{sec:prel} are satisfied.

We now provide a rigorous definition of the Riesz transforms $R^\alpha_j$,
$j=1,\ldots,d$, on $L^2(\R,w_{\alpha})$ by setting 
\begin{equation} \label{RRR}
R^{\alpha}_jf=\sum_{n\in\N} \frac{m(n_j,\alpha_j)}{\sqrt{2|n|+2|\alpha|+2d}}
\langle f,h_n^{\alpha} \rangle_{\alpha}\; h_{n-e_j}^{\alpha},
\end{equation}
for $f$ with the expansion $f=\sum_{n\in\N} \langle f,h_n^{\alpha}\rangle_{\alpha}h_n^{\alpha}$. 
Obviously, $R^{\alpha}_j$ is bounded on $L^2(\mathbb{R}^d, w_\alpha)$.

For $j=1,\ldots,d$, we define the Riesz kernels $R_j(x,y)$ as
\begin{equation} \label{riesz_ker}
R_j^{\alpha}(x,y) = \frac{1}{\sqrt{\pi}} \int_0^{\infty} 
	\delta_{j,x} G_t^{\alpha}(x,y) t^{-1\slash 2} \, dt.
\end{equation}
The following result shows that the kernel $R_j^{\alpha}(x,y)$ is associated,
in the Calder\'on-Zygmund theory sense, with the operator $R_j^{\alpha}$ defined in 
\eqref{RRR}.
\begin{propo}  \label{sto}
Assume that $\alpha=(\alpha_1,\dots,\alpha_d)$ is a multi-index
such that $\alpha_i\ge-1/2$ and $R_j^{\alpha}$ is defined by $\eqref{RRR}$.
Let $f,g\in C^\infty_{c,\mathbb{Z}_2^d}({\R})$ have disjoint supports. Then
\begin{equation}
\langle R_j^{\alpha}f,g\rangle_{\alpha}=\int_{{\R}}\int_{{\R}}
R_j^{\alpha}(x,y)f(y)\overline{g(x)}\,dw_{\alpha}(y) \, dw_{\alpha}(dx). \label{12c}
\end{equation}
\end{propo}
\begin{proof} 
Let $f=\sum_{n \in \N} a_n^\alpha h_n^\alpha$,
$g=\sum_{n \in \N} b_{n}^{\alpha} h_{n}^{\alpha}$. Then
$$
R_j^\alpha f= \sum_{n\in\N} \frac{m(n_j,\alpha_j)}{\sqrt{2|n|+2|\alpha|+2d}}
a_n^{\alpha} h_{n-e_j}^{\alpha}
$$
(convergence of the above series is in $L^2$) and, by Parseval's identity,
\begin{equation}
\langle R_j^{\alpha}f,g\rangle_{\alpha} = \sum_{n\in\N}
\frac{m(n_j,\alpha_j)}{\sqrt{2|n|+2|\alpha|+2d}} a_n^{\alpha} \overline{b_{n-e_j}^{\alpha}}.
\label{13c}
\end{equation}
To finish the proof it is now sufficient to observe that the right sides of \eqref{12c} 
and \eqref{13c}
coincide. This follows by noting that the proof of \eqref{gr} in Theorem \ref{ker_est} below
contains a proof of the slightly stronger estimate (see Section \ref{sec:kernel})
$$
\int_0^\infty \big|\delta_{j,x}G_t^{\alpha}(x,y)\big|t^{-1/2}dt\lesssim 
\frac{1}{w_{\alpha}(B(x,\|y-x\|))}
$$
and thus the assumption made on the supports of $f$ and $g$ implies
$$
\int_{{\R}}\int_{{\R}}\int_0^\infty \big|\delta_{j,x} G_t^\alpha(x,y)\big|
t^{-1/2}dt\,|\overline{g(x)}f(y)|\,dw_{\alpha}(y)\,dw_{\alpha}(x)<\infty.
$$
At this point the desired conclusion easily follows by repeating arguments from the proof of
\cite[Proposition 3.2]{ST1}.
\end{proof}

The next theorem is the heart of the matter: 
it says that the kernels of Riesz transforms
associated with the Dunkl harmonic oscillator satisfy standard estimates in the sense of the
homogeneous space $(\R,w_{\alpha},\|\cdot\|)$.
The corresponding proof is rather long and tricky, hence we decided to 
locate it in a separate unit (Section \ref{sec:kernel} below).
\begin{theor} \label{ker_est}
Let $\alpha \in [-1\slash 2, \infty)^d$. Then the kernels $R_j^{\alpha}(x,y)$, $j=1,\ldots,d$,
satisfy
\begin{align}
|R_j^{\alpha}(x,y)| & \lesssim \frac{1}{w_{\alpha}(B(x,\|y-x\|))}, \qquad x \neq y, \label{gr}\\
\|\nabla_{\! x,y} R_j^{\alpha}(x,y)\| & \lesssim \frac{1}{\|x-y\|} 
	\frac{1}{w_{\alpha}(B(x,\|y-x\|))}, \qquad x \neq y. \nonumber 
\end{align}
\end{theor}

It is well known that the classical weighted Calder\'on-Zygmund theory works, 
with appropriate adjustments,
when the underlying space is of homogeneous type.
Thus we shall use properly adjusted facts from the classic
Calder\'on-Zygmund theory (presented, for instance, in \cite{Duo}) 
in the setting of the homogeneous space $(\R,w_{\alpha},\|\cdot\|)$ 
without further comments.

For $1\le p < \infty$, we denote by
$A^{\alpha}_p = A^{\alpha}_p(\R,w_{\alpha})$ the Muckenhoupt class of $A_p$
weights related to the space $(\R,w_{\alpha},\|\cdot\|)$. More precisely,
$A^{\alpha}_p$ is the class of all nonnegative functions
$W \in L^1_{\textrm{loc}}(\R,w_{\alpha})$ such that
$W^{-1\slash (p-1)} \in L^1_{\textrm{loc}}(\R,w_{\alpha})$ and
\begin{equation*} 
\sup_{B \in \mathcal{B}} \bigg[ \frac{1}{w_{\alpha}(B)} \int_B W(x) \,
     dw_{\alpha}(x) \bigg]  \bigg[ \frac{1}{w_{\alpha}(B)} \int_B W(x)^{-1\slash (p-1)}
     \, dw_{\alpha}(x) \bigg]^{p-1} < \infty,
\end{equation*}
when $1<p<\infty$, or
\begin{equation*} 
\sup_{B \in \mathcal{B}}\, \frac{1}{w_{\alpha}(B)} \int_B W(x) \,
     dw_{\alpha}(x) \; \essup_{x \in B} \frac{1}{W(x)} < \infty,
\end{equation*}
if $p=1$; here $\mathcal{B}$ is the class of all Euclidean balls in $\R$.
It is easy to check that (in dimension one) the power weight $W_r(x) = |x|^{r}$ belongs
to $A^{\alpha}_p$, $1<p<\infty$, if and only if $-(2\alpha+2)<r<(2\alpha+2)(p-1)$,
and $W_r \in A^{\alpha}_1$ if and only if $-(2\alpha+2)< r \le 0$.
\begin{theor}  \label{thm:main}
Assume that $\alpha=(\alpha_1,\dots,\alpha_d)$ is a multi-index
such that $\alpha_i\ge-1/2$, $i=1,\ldots,d$. 
Then the Riesz operators $R^\alpha_j$,
$j=1,\ldots,d$, defined on $L^2(\R,w_{\alpha})$ by $(\ref{RRR})$, 
are Calder\'on-Zygmund operators associated with the kernels
$R^\alpha_j(x,y)$ defined by $(\ref{riesz_ker})$. In consequence, 
$R^\alpha_j$ extend uniquely to bounded linear operators on $L^p(\R,W dw_{\alpha})$,
$1<p<\infty$, $W\in A^{\alpha}_p$, and to bounded linear operators 
from $L^1(\R,W dw_{\alpha})$ to $L^{1,\infty}(\R,W dw_{\alpha})$, $W\in A_1$. Denoting
these extensions by the same symbols $R^\alpha_j$, for $f\in L^p(\R,Wdw_{\alpha})$, $1<p<\infty$, 
$W\in A_p$, we have
\begin{equation}  \label{star}
\langle R^\alpha_jf,h^{\alpha}_{n-e_j}\rangle_{\alpha} =
\frac{m(n_j,\alpha_j)}{\sqrt{2|n|+2|\alpha|+2d}}
\langle f,h_n^{\alpha}\rangle_{\alpha}.
\end{equation}
\end{theor}
\begin{proof}
The first statement is just a combination of Theorem \ref{ker_est} and Proposition \ref{sto}.
The second is a consequence of the general theory, see \cite{Duo}. 
To verify (\ref{star}) fix $j=1,\ldots,d$,
and choose $f_k\in L^p(\R,Wdw_{\alpha})\cap L^2(\R,w_{\alpha})$ such that 
$f_k\to f$ in $L^p(\R,Wdw_{\alpha})$. Then, for any $n\in\N,$
$\langle f_k,h_n^{\alpha}\rangle_{\alpha} \to \langle f,h_n^{\alpha}\rangle_{\alpha}$ and
$\langle R_j^\alpha f_k,h_{n-e_j}^{\alpha}\rangle_{\alpha}\to \langle R_j^\alpha f,
h_{n-e_j}^{\alpha}\rangle_{\alpha}$, where $R_j^\alpha f$ is, by the very definition, the limit of
$R_j^\alpha f_k$ in $L^p(\R,Wdw_{\alpha})$. The claim follows.
\end{proof}

We finish this section with an application of Theorem \ref{thm:main} connected with 
\emph{a priori} $L^p$-bounds. This kind of estimates is a useful tool in the theory of partial 
differential equations.
\begin{propo} \label{apriori}
Let $\alpha \in [-1\slash 2,\infty)^d$, $1<p<\infty$, and $i,j \in \{1,\ldots,d\}$. Then
$$
\| \delta_i^* \delta_j f \|_{L^p(\R,w_{\alpha})} \lesssim 
	\| L_{\alpha} f\|_{L^p(\R,w_{\alpha})}, \qquad f \in C^{\infty}_c(\R).
$$
\end{propo}

\begin{proof}
Let $\widehat{R}_j^{\alpha}$ be the adjoint of $R_j^{\alpha}$ taken in $L^2(\R,w_{\alpha})$.
Then
$$
\widehat{R}_j^{\alpha}f = \sum_{n \in \N} \frac{m(n_j+1,\alpha_j)}{\sqrt{2|n|+2|\alpha|+2d+2}}
	\langle f, h_n^{\alpha}\rangle_{\alpha} \, h^{\alpha}_{n + e_j}, \qquad f \in L^2(\R,w_{\alpha}).
$$
By Theorem \ref{thm:main} and duality, $\widehat{R}_j^{\alpha}$ extends to a bounded operator
in $L^p(\R,w_{\alpha})$, $1<p<\infty$. 
Therefore, to finish the proof, in suffices to show that
$$
\delta_i^* \delta_j f = \widehat{R}_i^{\alpha} R_j^{\alpha} L_{\alpha} f
$$
for $f \in C_c^{\infty}(\R)$. However, since such functions can be suitably approximated 
by linear combinations of $h_n^{\alpha}$, it is enough to verify the relation for $f = h_n^{\alpha}$.

Using the identity
$$
\delta_i^{*} h_n^{\alpha} = m(n_i+1,\alpha_i) h_{n+e_i}^{\alpha}, \qquad n \in \N,
$$
which, taking into account that $\delta_i^* = -\delta_i+2x_i$, can be obtained with the aid of 
\eqref{nn1}, \eqref{nn2} and \eqref{nn3}, we get
$$
\delta_i^* \delta_j h_n^{\alpha} = m(n_j,\alpha_j)\, \delta_i^* h^{\alpha}_{n-e_j}
	= m(n_j,\alpha_j) \,m\big((n-e_j)_i+1,\alpha_i\big)\, h^{\alpha}_{n-e_j+e_i}.
$$
On the other hand,
\begin{align*}
\widehat{R}_i^{\alpha} R_j^{\alpha} L_{\alpha} h_n^{\alpha} & = (2|n|+2|\alpha|+2d)\,
	\widehat{R}_i^{\alpha} \bigg( \frac{m(n_j,\alpha_j)}{\sqrt{2|n|+2|\alpha|+2|d|}} h_{n-e_j}\bigg) \\
	& = (2|n|+2|\alpha|+2d) \frac{m(n_j,\alpha_j)}{\sqrt{2|n|+2|\alpha|+2|d|}} 
		\frac{m((n-e_j)_i+1,\alpha_i)}{\sqrt{2|n-e_j|+2|\alpha|+2d+2}}\, h^{\alpha}_{n-e_j+e_i} \\
	& = m(n_j,\alpha_j)\, m\big((n-e_j)_i+1,\alpha_i\big) \, h^{\alpha}_{n-e_j+e_i}.
\end{align*}
The conclusion follows.
\end{proof}

\section{Kernel estimates for the $\mathbb{Z}_2^d$ Riesz operators} \label{sec:kernel}

This section is devoted to the proof of the standard estimates stated in Theorem \ref{ker_est}.
We begin with a suitable decomposition of the kernels under consideration:
$$
R_j^{\alpha}(x,y) = 
  \sum_{\varepsilon \in \{0,1\}^d} R_j^{\alpha,\varepsilon}(x,y), \qquad x,y \in \R,
$$
where the component kernels are given by
\begin{equation} \label{comp_ker}
R_j^{\alpha,\varepsilon}(x,y) = \frac{1}{\sqrt{\pi}} \int_0^{\infty} \delta_{j} 
	G_t^{\alpha,\varepsilon}(x,y) t^{-1\slash 2} \, dt;
\end{equation}
here and also later on $\delta_j$ always refers to the $x$ variable.
Clearly, it is sufficient to prove the estimates in question for each 
$R_j^{\alpha,\varepsilon}(x,y)$, $\varepsilon \in \{0,1\}^d$, separately.
Thus, from now on, we fix $\varepsilon \in \{0,1\}^d$.

Next, we claim that it is enough to show the relevant bounds 
only for $x,y \in \R_+$
(here, in fact, we could replace $\R_{+}$ by any other Weyl chamber of the considered group
of reflections).
Observe that $G^{\alpha,\varepsilon}_t(x,y)$, as a function
of either $x$ or $y$, is even with respect to the $j$th axis if $\varepsilon_j=0$ and is odd
with respect to the $j$th axis if $\varepsilon_j=1$.
The derivative $\delta_j$, when taken with respect to $x$ and restricted to even functions
in the $j$th coordinate, has the form
\begin{equation} \label{der_e}
\delta_j^e = \frac{\partial}{\partial x_j} + x_j,
\end{equation}
whereas the restriction of $\delta_j$ to odd functions in the $j$th coordinate is
\begin{equation} \label{der_o}
\delta_j^o = \frac{\partial}{\partial x_j} + x_j + \frac{2\alpha_j+1}{x_j}.
\end{equation}
Thus $\delta_j G_t^{\alpha,\varepsilon}(x,y)$, as a function of $x$, is either even or odd
with respect to the $j$th axis, depending on whether $\varepsilon_j=1$ or $\varepsilon_j=0$,
respectively.
Moreover, as a function of $y$, it is either even or odd with respect to the $j$th axis, depending
on whether $\varepsilon_j = 0$ or $\varepsilon_j=1$, correspondingly. It follows that
$$
|R^{\alpha,\varepsilon}_j(x,y)| = |R_j^{\alpha,\varepsilon}(\eta x, \xi y)|, \qquad
	\eta,\xi \in \{-1,1\}^d.
$$
Similarly, we have
$$
\|\nabla_{\!x,y} R^{\alpha,\varepsilon}_j(x,y)\| = 
	\|\nabla_{\!x,y} R_j^{\alpha,\varepsilon}(\eta x, \xi y)\|, \qquad
	\eta,\xi \in \{-1,1\}^d.
$$
This, together with the symmetry of $w_{\alpha}$, 
$$
w_{\alpha}(x) = w_{\alpha}(\xi x), \qquad \xi \in \{-1,1\}^d,
$$ 
proves the claim.

Finally, we also observe that, after restricting $x$ and $y$ to $\R_+$, 
with no loss of generality we may replace
the Euclidean balls $B(x,\|y-x\|)$ in the relevant estimates by their intersections with $\R_+$,
denoted further by $B^+(x,\|y-x\|)$.

Summing up, we are reduced to proving the following.

\begin{lemma} \label{ker_est_+}
Assume that $\alpha \in [-1\slash 2,\infty)^d$, 
$\varepsilon \in \{0,1\}^d$ and $j \in \{1,\ldots,d\}$.
Then
\begin{align*}
|R_j^{\alpha,\varepsilon}(x,y)| & \lesssim \frac{1}{w_{\alpha}^+(B^+(x,\|y-x\|))}, \qquad 
	x,y \in \R_+, \quad x \neq y, \\
\|\nabla_{x,y} R_j^{\alpha,\varepsilon}(x,y)\| & \lesssim \frac{1}{\|x-y\|} 
	\frac{1}{w^+_{\alpha}(B^+(x,\|y-x\|))}, \qquad x,y \in \R_+, \quad x \neq y, 
\end{align*}
with $w_{\alpha}^+$ being the restriction of $w_{\alpha}$ to $\R_+$.
\end{lemma}

The proof of Lemma \ref{ker_est_+} relies on a generalization of the technique applied recently
by the authors in the Laguerre function setting \cite{NS3}.
Thus the main tool in performing necessary kernel estimates will be Schl\"afli's integral 
representation of Poisson's type for the modified Bessel function, 
\begin{equation*} 
I_{\nu}(z) = z^{\nu} \int_{-1}^1 \exp({-z s})\, \Pi_{\nu}(ds), \qquad
    z>0, \quad \nu \ge -\frac{1}{2},
\end{equation*}
where the measure $\Pi_{\nu}$ is determined by the density
$$
\Pi_{\nu}(ds) = \frac{(1-s^2)^{\nu-1\slash 2}ds}{\sqrt{\pi} 2^{\nu}\Gamma{(\nu+1\slash 2)}},
    \qquad s \in (-1,1), 
$$
when $\nu > -1\slash 2$, and in the limiting case of $\nu = -1\slash 2$, 
$$
\Pi_{-1\slash 2} = \frac{1}{\sqrt{2\pi}} \big( \eta_{-1} + \eta_1 \big),
$$
where $\eta_{-1}$ and $\eta_1$ denote point masses at $-1$ and $1$, respectively.

Consequently, for $\alpha \in[-1\slash 2,\infty)^d$, the kernel $G_t^{\alpha,\varepsilon}(x,y)$
may be expressed as
\begin{align*}
& G^{\alpha,\varepsilon}_t(x,y) = \\
& \frac{(xy)^{\varepsilon}}{2^d(\sinh (2t))^{d+|\alpha|+|\varepsilon|}}
	 \int_{[-1,1]^d}
    \exp\bigg({-\frac{1}{2} \coth (2t) \big(\|x\|^2 + \|y\|^2\big)}
    {-\sum_{i=1}^d \frac{x_i y_i s_i}{\sinh (2t)}}\bigg) \, \Pi_{\alpha+\varepsilon}(ds),
\end{align*}
where $\Pi_{\alpha}$ denotes the product measure $\bigotimes_{i=1}^d \Pi_{\alpha_i}$. 
Then the change of variable $\!^{\dag}$
\footnotetext{$^{\dag}$ This useful change of variable was invented by Stefano Meda,
	in the context of the Ornstein-Uhlenbeck operator.}
\begin{equation} \label{tzeta}
t = t(\zeta) = \frac{1}{2} \log \frac{1+ \zeta}{1-\zeta}, \qquad \zeta \in (0,1),
\end{equation}
leads to the following useful symmetric formula: 
\begin{equation*} 
G^{\alpha,\varepsilon}_t(x,y) =\frac{1}{2^d} 
	\Big( \frac{1-\zeta^2}{2\zeta}\Big)^{d + |\alpha| + |\varepsilon|} 
		(xy)^{\varepsilon} \int_{[-1,1]^d}
    \exp\Big(-\frac{1}{4\zeta} q_{+}(x,y,s) - \frac{\zeta}{4} q_{-}(x,y,s)\Big) 
    	\, \Pi_{\alpha+\varepsilon}(ds),
\end{equation*}
with
$$
q_{\pm}(x,y,s) = \|x\|^2 + \|y\|^2 \pm 2 \sum_{i=1}^d x_i y_i s_i.
$$
For the sake of brevity we shall often write shortly $q_{+}$ or $q_{-}$ omitting the arguments.

We pass to the Riesz kernels. By the change of variable \eqref{tzeta} the component kernels 
\eqref{comp_ker} can be expressed as
\begin{equation} \label{k_r}
R_j^{\alpha,\varepsilon}(x,y) = \int_{[-1,1]^d} \Pi_{\alpha + \varepsilon}(ds) \int_0^1
	\beta_{d,\alpha+\varepsilon}(\zeta) \, \delta_j \psi_{\zeta}^{\varepsilon}(x,y,s) \, d\zeta,
\end{equation}
where
$$
\beta_{d,\alpha}(\zeta) = \frac{\sqrt{2}}{2^d\sqrt{\pi}} \bigg( \frac{1-\zeta^2}{2\zeta}
	 \bigg)^{d+|\alpha|} \frac{1}{1-\zeta^2} \bigg( \log \frac{1+\zeta}{1-\zeta} \bigg)^{-1\slash 2}
$$
(this is, up to the constant factor $2^{-d}$, 
$\beta_{d,\alpha}^1(\zeta)$ defined in \cite[(5.4)]{NS3}), and
$$
\psi_{\zeta}^{\varepsilon}(x,y,s) = (xy)^{\varepsilon}
	\exp\Big(-\frac{1}{4\zeta} q_{+}(x,y,s) - \frac{\zeta}{4} q_{-}(x,y,s)\Big);
$$
passing with $\delta_j$ under the integral sign and then changing the order of integration
that was implicitly performed above will be justified in a moment.

To compute $\delta_j \psi_{\zeta}^{\varepsilon}(x,y,s)$,
the expression appearing in \eqref{k_r},
observe that the derivative $\delta_j$ may be replaced either by $\delta_j^e$ or
$\delta_j^o$, depending on whether $\varepsilon_j =0$ or $\varepsilon_j=1$, respectively;
see \eqref{der_e}, \eqref{der_o} and the accompanying comments.
Then one easily arrives at
\begin{align} \label{d_psi}
\delta_j \psi_{\zeta}^{\varepsilon}(x,y,s)  = &
 \bigg[ (xy)^{\varepsilon} \Big( x_j - \frac{1}{2\zeta}(x_j+y_j s_j)
	- \frac{\zeta}{2} (x_j-y_j s_j) \Big)\\
	&  + \chi_{\{\varepsilon_j=1\}} (2\alpha_j+2) 
		y_j (xy)^{\varepsilon-e_j} \bigg] \exp\Big( - \frac{1}{4\zeta} q_+ - \frac{\zeta}{4} q_{-} \Big).
		\nonumber
\end{align}

We now come back to explaining the possibility of exchanging $\delta_{j}$
with the integral sign. 
Observe that integrating the expression
$$
\int_{[-1,1]^d} \Pi_{\alpha+\varepsilon}(ds) \; \partial_{x_j} \bigg[ (xy)^{\varepsilon}
	\exp\Big(-\frac{1}{4\zeta} q_{+}(x,y,s) -\frac{\zeta}{4} q_{-}(x,y,s)\Big)\bigg]
$$
with respect to $x_j$ in an interval $[0,v]$, $0<v<\infty$, leads to an absolutely
convergent double integral. Thus we may apply Fubini's theorem and then differentiate
the resulting equality in $v$. This gives the desired conclusion.

The application of Fubini's theorem that was necessary to get \eqref{k_r} is also justified since,
in fact, the proof (to be given below) of the first estimate in Lemma \ref{ker_est_+} contains
the proof of 
$$
\int_{[-1,1]^d} \Pi_{\alpha + \varepsilon}(ds) \int_0^1
	\beta_{d,\alpha+\varepsilon}(\zeta) \, |\delta_j \psi_{\zeta}^{\varepsilon}(x,y,s)| \, d\zeta
		< \infty.
$$

For proving Lemma \ref{ker_est_+} we will need several simple technical results obtained 
in \cite[Section 5]{NS3}. We gather them in the lemma below for the sake of reader's convenience.

\begin{lemma}[{\cite[Corollary 5.2, Lemma 5.5]{NS3}}] \label{m_lem}
Assume that $\alpha \in [-1\slash 2,\infty)^d$.
Let $b \ge 0$ and $c>0$ be fixed. Then, for any $j=1,\ldots,d$, we have
\begin{align*}
& \emph{(a)} \qquad
	\big( |x_j + y_j s_j| + |y_j + x_j s_j| \big)^b \exp\Big( - c\frac{1}{\zeta} q_{+}(x,y,s) \Big)
     \lesssim \zeta^{b\slash 2},\\
& \emph{(b)} \qquad
	\big( |x_j - y_j s_j| + |y_j - x_j s_j| \big)^b \exp\Big( - {c\zeta } q_{-}(x,y,s) \Big)
     \lesssim \zeta^{-b\slash 2},\\
& \emph{(c)} \qquad
(x_j)^b \exp\Big( -c \frac{1}{\zeta}q_{+}(x,y,s) -c \zeta q_{-}(x,y,s) \Big)
    \lesssim \zeta^{-b\slash 2}, \\
& \emph{(d)} \qquad
\int_0^1 \beta_{d,\alpha}(\zeta) \zeta^{-b-1\slash 2}
    \exp\Big( -c \frac{1}{\zeta} q_{+}(x,y,s) \Big) \, d\zeta
        \lesssim \big(q_{+}(x,y,s)\big)^{-d-|\alpha|-b},
\end{align*}
uniformly in $x,y \in \R_+$, $s \in [-1,1]^d$ and, except of $(d)$, in $\zeta \in (0,1)$. 
\end{lemma}

We will also need the following generalization of \cite[Proposition 5.9]{NS3}.

\begin{lemma} \label{lemhom}
Assume that $\alpha \in [-1\slash 2, \infty)^d$ and let $\gamma,\kappa \in [0,\infty)^d$ be fixed. 
If a complex-valued kernel $K(x,y)$ 
defined on $\R_{+} \times \R_{+} \backslash \{(x,y) : x = y\}$
satisfies
\begin{align*}
|K(x,y)| & \lesssim (x+y)^{2\gamma} \int_{[-1,1]^d}\Pi_{\alpha+\gamma+\kappa}(ds) 
	\; \big(q_{+}(x,y,s)\big)^{-d - |\alpha| -|\gamma|}, 
\end{align*}
then also 
\begin{align*} 
|K(x,y)| & \lesssim \frac{1}{w^+_{\alpha}(B^+(x,\|y-x\|))}, \qquad x \neq y. 
\end{align*}
Similarly, the estimate
\begin{align*}
\|\nabla_{\! x,y}K(x,y)\| & \lesssim (x+y)^{2\gamma}\int_{[-1,1]^d} \Pi_{\alpha+\gamma+\kappa}(ds)\;
    \big(q_{+}(x,y,s)\big)^{-d - |\alpha| -|\gamma| - 1\slash 2} 
\end{align*}
implies
\begin{align*} 
\|\nabla_{\! x,y} K(x,y)\| & \lesssim \frac{1}{\|x-y\|} \;
 \frac{1}{w^+_{\alpha}(B^+(x,\|y-x\|))}, \qquad x \neq y.
\end{align*}
\end{lemma}

\begin{proof}
When $\kappa = (0,\ldots,0)$, the proof is a direct modification of the arguments used in 
\cite{NS3} for showing Proposition 5.9 there. The difference is that now 
\cite[Lemma 5.8]{NS3} must be applied with different parameters, but otherwise the reasoning is
essentially the same. Thus we omit the details.

The general case of arbitrary $\kappa$ follows easily from the following observation:
for any fixed $u,v > 0$ and $j=1,\ldots,d$, 
$$
\int_{[-1,1]^d} \Pi_{\alpha+\gamma+u e_j}(ds) \, (q_+(x,y,s))^{-v} \lesssim
\int_{[-1,1]^d} \Pi_{\alpha+\gamma}(ds) \, (q_+(x,y,s))^{-v}.
$$
Thus it remains to justify the above estimate.

When $\alpha_j+\gamma_j>-1\slash 2$ the relation between the (one-dimensional)
densities,
$$
\Pi_{\alpha_j+\gamma_j+u}(s_j) \lesssim \Pi_{\alpha_j+\gamma_j}(s_j), \qquad s_j \in (-1,1),
$$
deduced immediately from the exact formula, does the job. 
In the case $\alpha_j+\gamma_j=-1\slash 2$ it is enough to notice that $(q_+)^{-v}$, as
a one-dimensional function of $s_j \in [-1,1]$, is decreasing; in particular, it attains its
maximum at the endpoint $s_j=-1$.
Then the conclusion follows by the special form of $\Pi_{-1\slash 2}$ and the finiteness of
the measure $\Pi_{\alpha_j+\gamma_j+u}$.
\end{proof}

Finally, notice that for $\lambda \in [-1\slash 2,\infty)^d$, $j \in \{1,\ldots,d\}$ and $u\ge 0$,
\begin{equation} \label{mon_beta}
\beta_{d,\lambda+u e_j}(\zeta) \le \zeta^{-u} \beta_{d,\lambda}(\zeta), 
	\qquad \zeta \in (0,1).
\end{equation}
This property follows directly from the exact formula for $\beta_{d,\alpha}$.

\begin{proof}[Proof of Lemma \ref{ker_est_+}: the growth estimate]
By \eqref{k_r} and \eqref{d_psi}
\begin{align*}
|R_j^{\alpha,\varepsilon}(x,y)|  \lesssim  &\; \int_{[-1,1]^d} \Pi_{\alpha+\varepsilon}(ds) \int_0^1
	\beta_{d,\alpha+\varepsilon}(\zeta) (xy)^{\varepsilon} \Big[x_j + \frac{1}{\zeta}|x_j+y_j s_j|
		+ \zeta |x_j-y_j s_j| \Big] \\
		& \quad \cdot \exp \Big( -\frac{1}{4\zeta} q_+ - \frac{\zeta}{4} q_{-} \Big) d\zeta \\
 & +  \; \chi_{\{\varepsilon_j=1\}} \int_{[-1,1]^d} \Pi_{\alpha+\varepsilon}(ds) \int_0^1
	\beta_{d,\alpha+\varepsilon}(\zeta) y_j(xy)^{\varepsilon-e_j} 
		\exp \Big( -\frac{1}{4\zeta} q_+ - \frac{\zeta}{4} q_{-} \Big) d\zeta \\
\equiv &\;  \mathcal{I}_1 + \mathcal{I}_2.
\end{align*}
We treat the integrals $\mathcal{I}_1$ and $\mathcal{I}_2$ separately.
Applying Lemma \ref{m_lem} (a), (b) and (c) we get
$$
\mathcal{I}_1 \lesssim (xy)^{\varepsilon} \int_{[-1,1]^d} \Pi_{\alpha+\varepsilon}(ds) \int_0^1
	\beta_{d,\alpha+\varepsilon}(\zeta) \zeta^{-1\slash 2} 
		\exp \Big( -\frac{1}{8\zeta} q_+ - \frac{\zeta}{8} q_{-} \Big) d\zeta.
$$
Next, using Lemma \ref{m_lem} (d) and observing that 
$(xy)^{\varepsilon} \le (x+y)^{2\varepsilon}$ gives
$$
\mathcal{I}_1 \lesssim (x+y)^{2\varepsilon} \int_{[-1,1]^d} \Pi_{\alpha+\varepsilon}(ds)
	(q_+)^{-d-|\alpha|-|\varepsilon|}.
$$
Now Lemma \ref{lemhom}, taken with $\gamma = \varepsilon$ and $\kappa = (0,\ldots,0)$, 
provides the required growth bound for $\mathcal{I}_1$.

To estimate $\mathcal{I}_2$ we assume that $\varepsilon_j=1$ (otherwise $\mathcal{I}_2$ vanishes) and
use \eqref{mon_beta}, specified to
$\lambda = \alpha + \varepsilon - e_j\slash 2$ and $u=1\slash 2$, 
getting
$$
\mathcal{I}_2 \lesssim  \int_{[-1,1]^d} 
	\Pi_{\alpha+\varepsilon}(ds) \int_0^1
	\beta_{d,\alpha+\varepsilon-e_j \slash 2}(\zeta) (x+y)^{2\varepsilon-e_j} \zeta^{-1\slash 2}
		\exp \Big( -\frac{1}{4\zeta} q_+ - \frac{\zeta}{4} q_{-} \Big) d\zeta.
$$
Then Lemma \ref{m_lem} (d) shows that
$$
\mathcal{I}_2 \lesssim (x+y)^{2(\varepsilon-e_j\slash 2)} 
	\int_{[-1,1]^d} \Pi_{\alpha+\varepsilon}(ds) 
	(q_+)^{-d-|\alpha|-|\varepsilon-e_j\slash 2|}.
$$
Combining the above with Lemma \ref{lemhom}, applied with $\gamma = \varepsilon-e_j \slash 2$ and
$\kappa = e_j \slash 2$,
produces the relevant bound for $\mathcal{I}_2$. 
This finishes proving the growth estimate for $R^{\alpha,\varepsilon}_j(x,y)$. 
\end{proof}

\begin{proof}[Proof of Lemma \ref{ker_est_+}: the smoothness estimate]
By symmetry reasons, without any loss of generality we may focus on the case $j=1$.
Thus the kernel under consideration is
$$
R_1^{\alpha,\varepsilon}(x,y) = \int_{[-1,1]^d} \Pi_{\alpha + \varepsilon}(ds) \int_0^1
	\beta_{d,\alpha+\varepsilon}(\zeta) \, \delta_1 \psi_{\zeta}^{\varepsilon}(x,y,s) \, d\zeta,
$$
with $\delta_1 \psi_{\zeta}^{\varepsilon}(x,y,s)$ given by \eqref{d_psi}.
Our task is to estimate $\|\nabla_{\! x,y}R_1^{\alpha,\varepsilon}(x,y)\|$. However, since
\begin{align} \nonumber
\|\nabla_{\! x,y}R_1^{\alpha,\varepsilon}(x,y)\| & =  \bigg\|
	\int_{[-1,1]^d} \Pi_{\alpha + \varepsilon}(ds) \int_0^1
	\beta_{d,\alpha+\varepsilon}(\zeta) \, \nabla_{\! x,y}\big[\delta_1 
		\psi_{\zeta}^{\varepsilon}(x,y,s)\big] \, d\zeta \bigg\| \\
& \lesssim \int_{[-1,1]^d} \Pi_{\alpha + \varepsilon}(ds) \int_0^1
	\beta_{d,\alpha+\varepsilon}(\zeta) \, \big\| \nabla_{\!x,y} \big[
		\delta_1 \psi_{\zeta}^{\varepsilon}(x,y,s)\big] \big\| \, d\zeta \label{i_est}
\end{align}
(passing with $\nabla_{\! x,y}$ under the integral signs is legitimate, the
justification being implicitly contained in the estimates that follow; 
see comments at the end of this proof),
and since only the first variables $x_1,y_1$ are distinguished, it is sufficient to show
a suitable bound for the double integral in \eqref{i_est}, but with the gradient restricted only
to the first two coordinates in $x$ and $y$.

To proceed, we need to compute the relevant derivatives. Denoting
$\Psi_{\zeta}^{\alpha,\varepsilon} = \delta_1 \psi_{\zeta}^{\varepsilon}(x,y,s)$
and abbreviating the exponential entering $\psi_{\zeta}^{\varepsilon}(x,y,s)$ to 
$\e(\zeta,q_{\pm})$ we get
\begin{align*}
\frac{\partial}{\partial x_1} \Psi_{\zeta}^{\alpha,\varepsilon} = &
	\bigg[ \chi_{\{\varepsilon_1=1\}} y_1 (xy)^{\varepsilon-e_1} 
	\Big( x_1 - \frac{1}{2\zeta}(x_1+y_1 s_1)
		- \frac{\zeta}{2} (x_1-y_1 s_1)\Big) \\ 
	& + (xy)^{\varepsilon} \Big( 1 - \frac{1}{2\zeta} - \frac{\zeta}{2} \Big) \bigg] \e(\zeta,q_{\pm})
	- \Big( \frac{1}{2\zeta}(x_1+y_1 s_1)+ \frac{\zeta}{2}(x_1-y_1 s_1) \Big)
		 \Psi_{\zeta}^{\alpha,\varepsilon},\\
\frac{\partial}{\partial y_1} \Psi_{\zeta}^{\alpha,\varepsilon} = &
	\bigg[ \chi_{\{\varepsilon_1=1\}} x_1 (xy)^{\varepsilon-e_1} 
	\Big( x_1 - \frac{1}{2\zeta}(x_1+y_1 s_1)
		- \frac{\zeta}{2} (x_1-y_1 s_1)\Big)  
	 - (xy)^{\varepsilon} \Big( \frac{s_1}{2\zeta} - \frac{\zeta s_1}{2} \Big) \\
	& + 2(\alpha_1+1) \chi_{\{\varepsilon_1=1\}} (xy)^{\varepsilon-e_1} \bigg] \e(\zeta,q_{\pm})
	- \Big( \frac{1}{2\zeta}(y_1+x_1 s_1)+ \frac{\zeta}{2}(y_1-x_1 s_1) \Big)
		 \Psi_{\zeta}^{\alpha,\varepsilon},\\
\frac{\partial}{\partial x_2} \Psi_{\zeta}^{\alpha,\varepsilon} = &
	\bigg[ \chi_{\{\varepsilon_2=1\}} y_2 (xy)^{\varepsilon-e_2} 
	\Big( x_1 - \frac{1}{2\zeta}(x_1+y_1 s_1)
		- \frac{\zeta}{2} (x_1-y_1 s_1)\Big)  
	  + 2(\alpha_1+1)\\ & \cdot \chi_{\{\varepsilon_1=\varepsilon_2=1\}} y_1 y_2
	   (xy)^{\varepsilon-e_1-e_2} \bigg] \e(\zeta,q_{\pm})
	- \Big( \frac{1}{2\zeta}(x_2+y_2 s_2)+ \frac{\zeta}{2}(x_2-y_2 s_2) \Big)
		 \Psi_{\zeta}^{\alpha,\varepsilon},\\
\frac{\partial}{\partial y_2} \Psi_{\zeta}^{\alpha,\varepsilon} = &
	\bigg[ \chi_{\{\varepsilon_2=1\}} x_2 (xy)^{\varepsilon-e_2} 
	\Big( x_1 - \frac{1}{2\zeta}(x_1+y_1 s_1)
		- \frac{\zeta}{2} (x_1-y_1 s_1)\Big)  
	  + 2(\alpha_1+1) \\& \cdot \chi_{\{\varepsilon_1=\varepsilon_2=1\}} y_1 x_2
	   (xy)^{\varepsilon-e_1-e_2} \bigg] \e(\zeta,q_{\pm})
	- \Big( \frac{1}{2\zeta}(y_2+x_2 s_2)+ \frac{\zeta}{2}(y_2-x_2 s_2) \Big)
		 \Psi_{\zeta}^{\alpha,\varepsilon}.		 
\end{align*}
We will estimate separately the four integrals resulting from replacing $\nabla_{\! x,y}$ in
\eqref{i_est} by one of the above derivatives. 
Denote these integrals by
$\mathcal{J}_{x_1}$, $\mathcal{J}_{y_1}$, $\mathcal{J}_{x_2}$ and $\mathcal{J}_{y_2}$, 
respectively, and assume that $\varepsilon_1=\varepsilon_2=1$ (otherwise the situation
is even easier since some terms do not appear).

Applying (a), (b) and (c) of Lemma \ref{m_lem} we get
\begin{align*}
\mathcal{J}_{x_1} & \lesssim 
	(x+y)^{2\varepsilon-e_1}
\int_{[-1,1]^d} \Pi_{\alpha + \varepsilon}(ds) \int_0^1
	\beta_{d,\alpha+\varepsilon}(\zeta) \zeta^{-1\slash 2} 
		\sqrt{\e(\zeta,q_{\pm})}\,d\zeta \\
	& \quad + (x+y)^{2\varepsilon} \int_{[-1,1]^d} 
	\Pi_{\alpha + \varepsilon}(ds) \int_0^1 \beta_{d,\alpha+\varepsilon}(\zeta) \zeta^{-1} 
		\e(\zeta,q_{\pm}) \,d\zeta \\
	& \quad + \int_{[-1,1]^d} \Pi_{\alpha + \varepsilon}(ds) \int_0^1
	\beta_{d,\alpha+\varepsilon}(\zeta) \zeta^{-1\slash 2} 
		|\delta_1 \psi_{\zeta}^{\varepsilon}(x,y,s)| \frac{1}{\sqrt{\e(\zeta,q_{\pm})}} 
		\, d\zeta \\
& \equiv \mathcal{J}^1_{x_1} + \mathcal{J}^2_{x_1} + \mathcal{J}^3_{x_1}.
\end{align*}
To estimate $\mathcal{J}_{x_1}^1$ we use \eqref{mon_beta} and then 
Lemma \ref{m_lem} (d), obtaining
\begin{align*}
\mathcal{J}_{x_1}^1 & \lesssim 
	(x+y)^{2\varepsilon-e_1}
\int_{[-1,1]^d} \Pi_{\alpha + \varepsilon}(ds) \int_0^1
	\beta_{d,\alpha+\varepsilon-e_1\slash 2}(\zeta) \zeta^{-1} 
		\sqrt{\e(\zeta,q_{\pm})} \,   
		d\zeta \\
& \lesssim 
	(x+y)^{2(\varepsilon-e_1\slash 2)}
	\int_{[-1,1]^d} \Pi_{\alpha + \varepsilon}(ds) 
		(q_+)^{-d-|\alpha|-|\varepsilon-e_1\slash 2|-1\slash 2}.
\end{align*}
The expression $\mathcal{J}_{x_1}^2$ can be treated in the same way as $\mathcal{I}_1$
appearing in the proof of the growth estimate above. We get
$$
\mathcal{J}_{x_1}^2 \lesssim (x+y)^{2\varepsilon}
	\int_{[-1,1]^d} \Pi_{\alpha + \varepsilon}(ds) 
		(q_+)^{-d-|\alpha|-|\varepsilon|-1\slash 2}.
$$
Finally, estimating $\mathcal{J}_{x_1}^3$ is completely analogous to estimating the growth of 
$|R_j^{\alpha,\varepsilon}(x,y)|$ performed earlier. The result is
\begin{align*}
\mathcal{J}_{x_1}^3 & \lesssim (x+y)^{2\varepsilon}
	\int_{[-1,1]^d} \Pi_{\alpha + \varepsilon}(ds) 
		(q_+)^{-d-|\alpha|-|\varepsilon|-1\slash 2} \\
	& \quad + 
	(x+y)^{2(\varepsilon-e_1\slash 2)}
	\int_{[-1,1]^d} \Pi_{\alpha + \varepsilon}(ds) 
		(q_+)^{-d-|\alpha|-|\varepsilon-e_1\slash 2|-1\slash 2}.
\end{align*}
Combining the above estimates of $\mathcal{J}^1_{x_1}$, $\mathcal{J}^2_{x_1}$ and
$\mathcal{J}^3_{x_1}$ with Lemma \ref{lemhom}, specified to either $\gamma = \varepsilon$
and $\kappa = (0,\ldots,0)$ or $\gamma = \varepsilon - e_1 \slash 2$ and $\kappa = e_1\slash 2$, 
provides the required smoothness bound for 
$\mathcal{J}_{x_1}$.

Considering $\mathcal{J}_{y_1}$, items (a), (b) and (c) of Lemma \ref{m_lem} lead to
$$
\mathcal{J}_{y_1} \lesssim \mathcal{J}^1_{x_1} + \mathcal{J}^2_{x_1} + \mathcal{J}^3_{x_1}
	+ 
		(x+y)^{2\varepsilon-2e_1} \int_{[-1,1]^d} 
	\Pi_{\alpha + \varepsilon}(ds) \int_0^1 \beta_{d,\alpha+\varepsilon}(\zeta)
		\e(\zeta,q_{\pm})  
		\, d\zeta.
$$
Thus it suffices to bound suitably the last term of the sum, which we denote by $\mathcal{J}_{y_1}^1$.
Making use of \eqref{mon_beta} and then applying item (d) of Lemma \ref{m_lem} gives
\begin{align*}
\mathcal{J}_{y_1}^1 & \lesssim 
		(x+y)^{2(\varepsilon-e_1)} \int_{[-1,1]^d} 
	\Pi_{\alpha + \varepsilon}(ds) \int_0^1 \beta_{d,\alpha+\varepsilon-e_1}(\zeta) \zeta^{-1}
		\e(\zeta,q_{\pm}) 
		\, d\zeta \\
& \lesssim  
		(x+y)^{2(\varepsilon-e_1)} \int_{[-1,1]^d} 
	\Pi_{\alpha + \varepsilon}(ds) (q_+)^{-d-|\alpha|-|\varepsilon-e_1|-1\slash 2}.		
\end{align*}
Now Lemma \ref{lemhom} employed with $\gamma = \varepsilon -e_1$ and $\kappa = e_1$
implies the desired bound of $\mathcal{J}^1_{y_1}$.

Passing to $\mathcal{J}_{x_2}$, items (a), (b) and (c) of Lemma \ref{m_lem} reveal that
\begin{align*}
\mathcal{J}_{x_2} & \lesssim 
	(x+y)^{2\varepsilon-e_2}
\int_{[-1,1]^d} \Pi_{\alpha + \varepsilon}(ds) \int_0^1
	\beta_{d,\alpha+\varepsilon}(\zeta) \zeta^{-1\slash 2} 
		\sqrt{\e(\zeta,q_{\pm})}\,d\zeta \\
	& \quad + 
	(x+y)^{2\varepsilon-e_1-e_2} \int_{[-1,1]^d} 
	\Pi_{\alpha + \varepsilon}(ds) \int_0^1 \beta_{d,\alpha+\varepsilon}(\zeta) 
		\e(\zeta,q_{\pm}) \,d\zeta \\
	& \quad + \int_{[-1,1]^d} \Pi_{\alpha + \varepsilon}(ds) \int_0^1
	\beta_{d,\alpha+\varepsilon}(\zeta) \zeta^{-1\slash 2} 
		|\delta_1 \psi_{\zeta}^{\varepsilon}(x,y,s)| \frac{1}{\sqrt{\e(\zeta,q_{\pm})}}\, d\zeta \\
& \equiv \mathcal{J}^1_{x_2} + \mathcal{J}^2_{x_2} + \mathcal{J}^3_{x_2}.
\end{align*}
Treatment of $\mathcal{J}^1_{x_2}$ and $\mathcal{J}^3_{x_2}$ is the same as in case of
$\mathcal{J}^1_{x_1}$ and $\mathcal{J}^3_{x_1}$. Thus it remains to deal with $\mathcal{J}^2_{x_2}$.
Let $\gamma = \varepsilon -e_1\slash 2 -e_2\slash 2$.
In view of \eqref{mon_beta} and Lemma \ref{m_lem} (d) we have
\begin{align*}
\mathcal{J}_{x_2}^2 & \lesssim 
	(x+y)^{2\gamma}
	 \int_{[-1,1]^d} \Pi_{\alpha + \varepsilon}(ds) 
	  \int_0^1 \beta_{d,\alpha+\gamma}(\zeta) \zeta^{-1}
		\e(\zeta,q_{\pm})   
		\, d\zeta \\	
& \lesssim 
	(x+y)^{2\gamma}
	 \int_{[-1,1]^d} \Pi_{\alpha + \varepsilon}(ds) 
	(q_+)^{-d-|\alpha|-|\gamma|-1\slash 2}.
\end{align*}
This, together with Lemma \ref{lemhom} 
taken with $\kappa = e_1\slash 2 + e_2\slash 2$
provides the relevant estimate of $\mathcal{J}_{x_2}$.

Eventually, the analysis related to $\mathcal{J}_{y_2}$ is analogous to that for $\mathcal{J}_{x_2}$.

We now come back to explaining the possibility of exchanging $\partial_{x_1}$
with the integral over $[-1,1]^d \times (0,1)$ (a similar argument works for
the remaining derivatives, and
considering only $j=1$ does not affect the generality). To apply
Fubini's theorem and then use the argument invoked earlier in connection with \eqref{k_r},
it is sufficient to show that
$$
\int_u^v \bigg(\int_{[-1,1]^d} \Pi_{\alpha+\varepsilon}(ds) \int_0^1  \big| \partial_{x_1}
    \big[\delta_1 \psi^{\varepsilon}_{\zeta}(x,y,s) \big]\big| 
    	\beta_{d,\alpha+\varepsilon}(\zeta)\, d\zeta \bigg) dx_1
         < \infty,
$$
for any $0<u<v<\infty$ such that $y_1$ is not in $[u,v]$ when $x_i=y_i$, $2\le i \le d$.
This estimate, however, may be easily obtained with the aid of the bound for $\mathcal{J}_{x_1}$ 
established above.

The proof of the smoothness estimate in Lemma \ref{ker_est_+} is complete.
\end{proof}

\end{document}